\documentclass[reqno]{amsart}
\usepackage{hyperref}

\usepackage{amsmath,amssymb,amsthm}
\usepackage{enumerate}
\usepackage{color}
\usepackage{epsfig}
\usepackage{graphics}
\usepackage{graphicx}
\usepackage{subfigure}

\begin{document}
\title[]{parabolic non-automorphism induced toeplitz-composition c*-algebras with piece-wise quasi-continuous symbols}
\author[u. g\"{u}l]{u\u{g}ur g\"{u}l}

\address{u\u{g}ur g\"{u}l,  \newline
Hacettepe University, Department of Mathematics, 06800, Beytepe,
Ankara, TURKEY}
\email{\href{mailto:gulugur@gmail.com}{gulugur@gmail.com}}

\thanks{Date: 28/05/2012}

\begin{abstract}
In this paper we consider the C*-algebra
$C^{*}(\{C_{\varphi}\}\cup\mathcal{T}(PQC(\mathbb{T})))/K(H^{2})$
generated by Toeplitz operators with piece-wise quasi-continuous
symbols and a composition operator induced by a parabolic linear
fractional non-automorphism symbol modulo compact operators on the
Hilbert-Hardy space $H^{2}$. This C*-algebra is commutative. We
characterize its maximal ideal space. We apply our results to the
question of determining the essential spectra of linear
combinations of a class of composition operators and Toeplitz
operators.

\end{abstract}

\maketitle
\newtheorem{theorem}{Theorem}
\newtheorem{acknowledgement}[theorem]{Acknowledgement}
\newtheorem{algorithm}[theorem]{Algorithm}
\newtheorem{axiom}[theorem]{Axiom}
\newtheorem{case}[theorem]{Case}
\newtheorem{claim}[theorem]{Claim}
\newtheorem{conclusion}[theorem]{Conclusion}
\newtheorem{condition}[theorem]{Condition}
\newtheorem{conjecture}[theorem]{Conjecture}
\newtheorem{corollary}[theorem]{Corollary}
\newtheorem{criterion}[theorem]{Criterion}
\newtheorem{definition}[theorem]{Definition}
\newtheorem{example}[theorem]{Example}
\newtheorem{exercise}[theorem]{Exercise}
\newtheorem{lemma}[theorem]{Lemma}
\newtheorem{notation}[theorem]{Notation}
\newtheorem{problem}[theorem]{Problem}
\newtheorem{proposition}[theorem]{Proposition}
\newtheorem{remark}[theorem]{Remark}
\newtheorem{solution}[theorem]{Solution}
\newtheorem{summary}[theorem]{Summary}
\newtheorem*{thma}{Main Theorem}
\newtheorem*{thmc}{Theorem C}
\newtheorem*{thmd}{Theorem D}
\newcommand{\norm}[1]{\left\Vert#1\right\Vert}
\newcommand{\abs}[1]{\left\vert#1\right\vert}
\newcommand{\set}[1]{\left\{#1\right\}}
\newcommand{\Real}{\mathbb R}
\newcommand{\eps}{\varepsilon}
\newcommand{\To}{\longrightarrow}
\newcommand{\BX}{\mathbf{B}(X)}
\newcommand{\A}{\mathcal{A}}

\section*{introduction}
Toeplitz-composition C*-algebras are C*-algebras generated by the
shift operator $T_{z}$ and a composition operator $C_{\varphi}$
induced by a linear fractional self-map $\varphi$ of $\mathbb{D}$.
The structure of these C*-algebras rely heavily upon whether
$\varphi$ is an automorphism or not. The automorphism case was
investigated in detail by Jury in [Jur] whereas the
non-automorphism case was investigated by Kriete, MacCluer and
Moorhouse in [KMM]. However before [Jur] and [KMM], it was shown
by Bourdon, Levi, Narayan and Shapiro [BLNS] that for any
parabolic linear-fractional non-automorphism $\varphi$, the
self-commutator
$C_{\varphi}^{*}C_{\varphi}-C_{\varphi}C_{\varphi}^{*}$ of the
composition operator and the commutator
$T_{z}C_{\varphi}-C_{\varphi}T_{z}$ are compact on the Hardy space
$H^{2}$. Hence if $\varphi$ is a parabolic non-automorphism then
the Toeplitz composition C*-algebra is commutative modulo compact
operators. It is our aim in this paper to extend this result to
Toeplitz composition C*-algebras whose Toeplitz operators have a
larger class of symbols namely the piece-wise quasi-continuous
class $PQC(\mathbb{T})$.

For a given subalgebra $B\subseteq L^{\infty}(\mathbb{T})$, the
Toeplitz C*-algebra generated by Toeplitz operators with symbols
in $B$ is denoted by $\mathcal{T}(B)$ and is defined as
\begin{equation*}
\mathcal{T}(B)=C^{*}(\{T_{a}:a\in B\})
\end{equation*}
When $B=C(\mathbb{T})$ the algebra continuous functions on the
circle, Coburn [Co] showed that for any $a,b\in C(\mathbb{T})$,
$$T_{a}T_{b}-T_{ab}\in K(H^{2})$$ and hence
$\mathcal{T}(C(\mathbb{T}))/K(H^{2})$ is commutative. Coburn has
also shown that
\begin{equation*}
\mathcal{T}(C(\mathbb{T}))/K(H^{2})\cong C(\mathbb{T})
\end{equation*}
This means that for any $a\in C(\mathbb{T})$, $T_{a}$ is Fredholm
if and only if $a(\lambda)\neq 0$ $\forall\lambda\in\mathbb{T}$.
In this case the Fredholm index $ind(T_{a})$ is given by
\begin{equation*}
ind(T_{a})=-\frac{1}{2\pi i}\int_{a(\mathbb{T})}\frac{dz}{z}
\end{equation*}
Moreover one has $\forall a\in C(\mathbb{T})$ and $\forall b\in
L^{\infty}(\mathbb{T})$, $$T_{a}T_{b}-T_{ab}\in K(H^{2}).$$
Douglas [Dou] carried Coburn's result one step further by showing
that $\forall a\in H^{\infty}+C$ and $\forall b\in
L^{\infty}(\mathbb{T})$,
$$T_{a}T_{b}-T_{ab}\in K(H^{2}).$$ From this one deduces that
$\mathcal{T}(QC(\mathbb{T})/K(H^{2})$ is a commutative C*-algebra
and moreover
\begin{equation*}
\mathcal{T}(QC(\mathbb{T}))/K(H^{2})\cong QC(\mathbb{T})
\end{equation*}
where $QC(\mathbb{T})=(H^{\infty}+C)\cap\overline{(H^{\infty}+C)}$
is the algebra of quasi-continuous functions on $\mathbb{T}$. The
Toeplitz C*-algebra $\mathcal{T}(PC(\mathbb{T}))/K(H^{2})$ was
analyzed in detail by Gohberg and Krupnik in [GoK] where
$PC(\mathbb{T})$ is the algebra of piece-wise continuous functions
on $\mathbb{T}$ which is defined as follows:
\begin{equation*}
PC(\mathbb{T})=\{a\in
L^{\infty}(\mathbb{T}):\lim_{\theta\rightarrow 0^{+}}a(\lambda
e^{i\theta})=a(\lambda^{+}),\lim_{\theta\rightarrow
0^{-}}a(\lambda
e^{i\theta})=a(\lambda^{-})\quad\textrm{exist}\quad\forall\lambda\in\mathbb{T}\}
\end{equation*}
They showed that $\forall a,b\in PC(\mathbb{T})$ the commutator
$$T_{a}T_{b}-T_{b}T_{a}\in K(H^{2}),$$ however the semi-commutator
$$T_{a}T_{b}-T_{ab}\not\in K(H^{2})$$ unless $a$ and $b$ have no
common point of discontinuity. Hence
$\mathcal{T}(PC(\mathbb{T}))/K(H^{2})$ is a commutative C*-algebra
but is not isometrically isomorphic to $PC(\mathbb{T})$. Gohberg
and Krupnik also give an explicit description of the maximal ideal
space of $\mathcal{T}(PC(\mathbb{T}))/K(H^{2})$ as the cylinder
$\mathbb{T}\times [0,1]$ equipped with the topology consisting of
the subsets
\begin{eqnarray*}
& &
U(e^{i\varphi_{0}},0)=\{(e^{i\varphi},t):\varphi_{0}-\delta<\varphi<\varphi_{0},0\leq
t\leq 1\}\cup\{(e^{i\varphi_{0}},t):0\leq t\leq\varepsilon\}\\
& &
U(e^{i\varphi_{0}},1)=\{(e^{i\varphi},t):\varphi_{0}<\varphi<\varphi_{0}+\delta,0\leq
t\leq 1\}\cup\{(e^{i\varphi_{0}},t):1-\varepsilon\leq t\leq 1\}\\
& &
U(e^{i\varphi_{0}},t_{0})=\{(e^{i\varphi_{0}},t):t_{0}-\delta_{1}<t<t_{0}+\delta_{2}\}\\
\end{eqnarray*}
where the action of any $(\lambda,t)\in\mathbb{T}\times [0,1]$ on
any $[T_{u}]\in\mathcal{T}(PC(\mathbb{T}))/K(H^{2})$ with $u\in
PC(\mathbb{T})$ is given by
\begin{equation*}
(\lambda,t)([T_{u}])=tu(\lambda^{+})+(1-t)u(\lambda^{-}).
\end{equation*}
In [Sar1] Sarason examined the Toeplitz C*-algebra
$\mathcal{T}(PQC(\mathbb{T}))$ in detail where $PQC(\mathbb{T})$
is the algebra of functions generated by $PC(\mathbb{T})$ and
$QC(\mathbb{T})$. By the results of Gohberg-Krupnik and Douglas,
it is clear that $\mathcal{T}(PQC(\mathbb{T}))/K(H^{2})$ is a
commutative C*-algebra. In [Sar2] Sarason characterized the
maximal ideal space $M(\mathcal{T}(PQC(\mathbb{T}))/K(H^{2}))$ of
$\mathcal{T}(PQC(\mathbb{T}))/K(H^{2})$ as a certain subset of
$M(QC(\mathbb{T}))\times [0,1]$. In [Gul] we introduced a class of
operators on $H^{2}(\mathbb{D})$ called the ``Fourier
multipliers". For any $\vartheta\in C([0,\infty])$, the Fourier
multiplier $D_{\vartheta}:H^{2}(\mathbb{D})\rightarrow
H^{2}(\mathbb{D})$ is defined as
\begin{equation*}
D_{\vartheta}=\Phi^{-1}\circ\mathcal{F}^{-1}\circ
M_{\vartheta}\circ\mathcal{F}\circ\Phi
\end{equation*}
where $\Phi: H^{2}(\mathbb{D})\rightarrow H^{2}(\mathbb{H})$ is
the isometric isomorphism
\begin{equation*}
(\Phi f)(z)=\frac{1}{z+i}f\left(\frac{z-i}{z+i}\right),
\end{equation*}
$\mathcal{F}$ is the Fourier transform and
$M_{\vartheta}:L^{2}(\mathbb{R}^{+})\rightarrow
L^{2}(\mathbb{R}^{+})$ is the multiplication operator
\begin{equation*}
(M_{\vartheta}f)(t)=\vartheta(t)f(t)
\end{equation*}
We showed in [Gul] that for any $\vartheta\in C([0,\infty])$ and
$a\in QC(\mathbb{T})$ the commutator
\begin{equation*}
[T_{a},D_{\vartheta}]=T_{a}D_{\vartheta}-D_{\vartheta}T_{a}\in
K(H^{2}).
\end{equation*}
In her thesis, R. Schmitz [Sch] studied the C*-algebra generated
by Toeplitz operators with $PC(\mathbb{T})$ symbols and the
composition operator with parabolic linear-fractional
non-automorphism symbol. The symbol $\varphi_{a}$ of this
composition operator, where $a\in\mathbb{C}$ with $\Im(a)>0$,
looks like
\begin{equation*}
\varphi_{a}(z)=\frac{(2i-a)z+a}{-az+a+2i}
\end{equation*}
Since
\begin{equation}
C_{\varphi_{a}}=T_{2i-a-az}D_{\vartheta_{a}}
\end{equation}
where $\vartheta_{a}(t)=e^{iat}$, the C*-algebra that Schmitz
considered coincides with the C*-algebra generated by Toeplitz
operators with $PC(\mathbb{T})$ symbols and continuous Fourier
multipliers. Schmitz showed that this C*-algebra is commutative
modulo $K(H^{2})$ by showing that the commutator
$[T_{a},D_{\vartheta}]\in K(H^{2})$. Let
$$\Psi(PQC(\mathbb{T}),C([0,\infty])=C^{*}(\mathcal{T}(PQC(\mathbb{T}))\cup
F_{C([0,\infty])})$$ where
$F_{C([0,\infty])}=\{D_{\vartheta}:\vartheta\in C([0,\infty])\}$
is the set of all continuous Fourier multipliers. By equation (1)
we also have
$$\Psi(PQC(\mathbb{T}),C([0,\infty])=C^{*}(\mathcal{T}(PQC(\mathbb{T}))\cup\{C_{\varphi_{a}}\})$$
for $\Im(a)>0$. Then by [Gul] and [Sch],
$\Psi(PQC(\mathbb{T}),C([0,\infty])/K(H^{2})$ is a commutative
C*-algebra. In this work we describe the maximal ideal space of
this C*-algebra. In particular we prove the following result:
\begin{thma}
Let
$$\Psi(PQC(\mathbb{T}),C([0,\infty])=C^{*}(\mathcal{T}(PQC(\mathbb{T}))\cup
F_{C([0,\infty])})=C^{*}(\mathcal{T}(PQC(\mathbb{T}))\cup\{C_{\varphi_{a}}\})$$
be the C*-algebra generated by Toeplitz operators with piece-wise
quasi-continuous symbols and Fourier multipliers on the Hardy
space $H^{2}(\mathbb{D})$. Then\\
$\Psi(PQC(\mathbb{T}),C([0,\infty])/K(H^{2})$ is a commutative
C*-algebra and its maximal ideal space is characterized as
\begin{equation*}
M(\Psi)\cong (M_{1}(\mathcal{T}(PC))\times M_{1}(QC)\times
[0,\infty])\cup
([\cup_{\lambda\in\mathbb{T}}(M_{\lambda}(\mathcal{T}(PC))\times
M_{\lambda}(QC))]\times\{\infty\})
\end{equation*}
where $M_{\lambda}(\mathcal{T}(PC))=\{x\in
M(\mathcal{T}(PC)):x\mid_{C(\mathbb{T})}=\delta_{\lambda},\delta_{\lambda}(T_{f})=f(\lambda)\}$
and $M_{\lambda}(QC)=\{x\in
M(QC):x\mid_{C(\mathbb{T})}=\delta_{\lambda},\delta_{\lambda}(T_{f})=f(\lambda)\}$
are the fibers of $M(\mathcal{T}(PC))$ and $M(QC)$ at $\lambda$
respectively.
\end{thma}
  In [Gul] we showed that a certain class of composition operators
  acting on $H^{2}(\mathbb{D})$, called ``the quasi-parabolic"
  composition operators fall inside the C*-algebra \\
  $\Psi(QC(\mathbb{T}),C([0,\infty]))$. Using this result we
  determine the essential spectra of linear combinations of
  quasi-parabolic composition operators and Toeplitz operators
  with piece-wise quasi-continuous symbols by the above
  characterization of the maximal ideal space of
  $\Psi(PQC(\mathbb{T}),C([0,\infty])/K(H^{2})$.
\section{preliminaries}
In this section we fix the notation that we will use throughout
and recall some preliminary facts that will be used in the sequel.

Let $S$ be a compact Hausdorff topological space. The space of all
complex valued continuous functions on $S$ will be denoted by
$C(S)$. For any $f\in C(S)$, $\parallel f\parallel_{\infty}$ will
denote the sup-norm of $f$, i.e. $$\parallel
f\parallel_{\infty}=\sup\{\mid f(s)\mid:s\in S\}.$$ For a Banach
space $X$, $K(X)$ will denote the space of all compact operators
on $X$ and $\mathcal{B}(X)$  will denote the space of all bounded
linear operators on $X$. The open unit disc will be denoted by
$\mathbb{D}$, the open upper half-plane will be denoted by
$\mathbb{H}$, the real line will be denoted by $\mathbb{R}$ and
the complex plane will be denoted by $\mathbb{C}$. The one point
compactification of $\mathbb{R}$ will be denoted by
$\dot{\mathbb{R}}$ which is homeomorphic to $\mathbb{T}$. For any
$z\in$ $\mathbb{C}$, $\Re(z)$ will denote the real part, and
$\Im(z)$ will denote the imaginary part of $z$, respectively. For
any subset $S\subset$ $B(H)$, where $H$ is a Hilbert space, the
C*-algebra generated by $S$ will be denoted by $C^{*}(S)$. The
Cayley transform $\mathfrak{C}$ will be defined by
\begin{equation*}
\mathfrak{C}(z)=\frac{z-i}{z+i}.
\end{equation*}
For any $a\in$ $L^{\infty}(\mathbb{R})$ (or $a\in$
$L^{\infty}(\mathbb{T})$), $M_{a}$ will be the multiplication
operator on $L^{2}(\mathbb{R})$ (or $L^{2}(\mathbb{T})$) defined
as
\begin{equation*}
M_{a}(f)(x)=a(x)f(x).
\end{equation*}
For convenience, we remind the reader of the rudiments of Gelfand
theory of commutative Banach algebras and Toeplitz operators.

Let $A$ be a commutative Banach algebra. Then its maximal ideal
space $M(A)$ is defined as
\begin{equation*}
    M(A)=\{x\in A^{*}:x(ab)=x(a)x(b)\quad\forall a,b\in A\}
\end{equation*}
where $A^{*}$ is the dual space of $A$. If $A$ has identity then
$M(A)$ is a compact Hausdorff topological space with the weak*
topology. The Gelfand transform $\Gamma:A\rightarrow C(M(A))$ is
defined as
\begin{equation*}
    \Gamma(a)(x)=x(a).
\end{equation*}
 If $A$ is a commutative C*-algebra with
identity, then $\Gamma$ is an isometric *-isomorphism between $A$
and $C(M(A))$. If $A$ is a C*-algebra and $I$ is a two-sided
closed ideal of $A$, then the quotient algebra $A/I$ is also a
C*-algebra (see [Dav]).
 For $a\in A$ the spectrum $\sigma_{A}(a)$ of $a$ on $A$
is defined as
\begin{equation*}
    \sigma_{A}(a)=\{\lambda\in\mathbb{C}:\lambda e-a\ \ \textrm{is not invertible in}\ A\},
\end{equation*}
where $e$ is the identity of $A$. We will use the spectral
permanency property of C*-algebras (see [Rud], pp. 283 and [Dav],
pp.15); i.e. if $A$ is a C*-algebra with identity and $B$ is a
closed *-subalgebra of $A$, then for any $b\in B$ we have
\begin{equation}
\sigma_{B}(b)=\sigma_{A}(b).
\end{equation}
To compute essential spectra we employ the following important
fact (see [Rud], pp. 268 and [Dav], pp. 6, 7): If $A$ is a
commutative Banach algebra with identity then for any $a\in A$ we
have
\begin{equation}
    \sigma_{A}(a)=\{\Gamma(a)(x)=x(a):x\in M(A)\}.
\end{equation}
In general (for $A$ not necessarily commutative), we have
\begin{equation}
    \sigma_{A}(a)\supseteq\{x(a):x\in M(A)\}.
\end{equation}

For a Banach algebra $A$, we denote by $com(A)$ the closed ideal
in $A$ generated by the commutators
$\{a_{1}a_{2}-a_{2}a_{1}:a_{1},a_{2}\in A\}$. It is an algebraic
fact that the quotient algebra $A/com(A)$ is a commutative Banach
algebra. The reader can find detailed information about Banach and
C*-algebras in [Rud] and [Dav] related to what we have reviewed so
far.

The essential spectrum $\sigma_{e}(T)$ of an operator $T$ acting
on a Banach
  space $X$ is the spectrum of the coset of $T$ in the Calkin algebra
  $\mathcal{B}(X)/K(X)$, the algebra of bounded linear operators modulo
  compact operators. The well known Atkinson's theorem identifies the essential
  spectrum of $T$ as the set of all $\lambda\in$ $\mathbb{C}$ for
  which $\lambda I-T$ is not a Fredholm operator. The essential norm of $T$ will be denoted by $\parallel T\parallel_{e}$ which is defined as
\begin{equation*}
 \parallel T\parallel_{e}=\inf\{\parallel T+K\parallel:K\in K(X)\}
\end{equation*}
   The bracket $[\cdot]$ will denote the equivalence class modulo
  $K(X)$. An operator $T\in\mathcal{B}(H)$ is called essentially
  normal if $T^{*}T-TT^{*}\in K(H)$ where $H$ is a Hilbert space and
  $T^{*}$ denotes the Hilbert space adjoint of $T$.

  For $1\leq p < \infty$ the Hardy space of the unit disc will be
denoted by $H^{p}(\mathbb{D})$ and the Hardy space of the upper
half-plane will be denoted by $H^{p}(\mathbb{H})$.

  The two Hardy spaces $H^{2}(\mathbb{D})$ and $H^{2}(\mathbb{H})$
    are isometrically isomorphic. An isometric isomorphism $\Phi:H^{2}(\mathbb{D})\longrightarrow$ $H^{2}(\mathbb{H})$
is given by
   \begin{equation}\label{nice1}
    \Phi(g)(z)=
   \bigg(\frac{1}{\sqrt{\pi}(z+i)}\bigg)g\bigg(\frac{z-i}{z+i}\bigg)
   \end{equation}
The mapping $\Phi$ has an inverse
$\Phi^{-1}:H^{2}(\mathbb{H})\longrightarrow$ $H^{2}(\mathbb{D})$
given by
\begin{equation*}\label{nice2}
\Phi^{-1}(f)(z)= \frac{e^\frac{i\pi}{2}(4\pi)^\frac{1}{2}}{(1-z)}
f\bigg(\frac{i(1+z)}{1-z}\bigg)
\end{equation*}
For more details see [Hof, pp. 128-131].

The Toeplitz operator with symbol $a$ is defined as
    $$T_{a}=P M_{a}|_{H^{2}} ,$$
where $P$ denotes the orthogonal projection of $L^{2}$ onto
$H^{2}$. A good reference about Toeplitz operators on $H^{2}$ is
Douglas' treatise ([Dou]). Although the Toeplitz operators treated
in [Dou] act on the
 Hardy space of the unit disc, the results can be transfered
 to the upper half-plane case using the isometric isomorphism $\Phi$
 introduced by equation (5). In the sequel the following identity
 will be used:
\begin{equation}
     \Phi^{-1}\circ T_{a}\circ\Phi=T_{a\circ \mathfrak{C}^{-1}} ,
\end{equation}
where $a\in L^{\infty}(\mathbb{R})$. We also employ the fact
\begin{equation}
    \parallel T_{a}\parallel_{e}=\parallel
    T_{a}\parallel=\parallel a\parallel_{\infty}
\end{equation}
 for any $a\in L^{\infty}(\mathbb{R})$, which is a consequence
 of Theorem 7.11 of [Dou] (pp. 160--161) and equation (6). For any subalgebra $A\subseteq L^{\infty}(\mathbb{T})$ the Toeplitz C*-algebra generated
 by symbols in $A$ is defined to be
\begin{equation*}
 \mathcal{T}(A)=C^{*}(\{T_{a}:a\in A\}).
\end{equation*}
 It is a well-known result of Sarason (see [Sar2]) that the
 set of functions
\begin{equation*}
 H^{\infty}+C=\{f_{1}+f_{2}:f_{1}\in H^{\infty}(\mathbb{D}),f_{2}\in C(\mathbb{T})\}
\end{equation*}
 is a closed subalgebra of $L^{\infty}(\mathbb{T})$. The following theorem of
 Douglas [Dou] will be used in the sequel.
\begin{theorem} [\scshape Douglas' Theorem]
    Let $a$,$b\in$ $H^{\infty}+C$ then the semi-commutators
    $$T_{ab}-T_{a}T_{b}\in K(H^{2}(\mathbb{D})),\quad T_{ab}-T_{b}T_{a}\in K(H^{2}(\mathbb{D})), $$
    and hence the commutator
    $$[T_{a},T_{b}]=T_{a}T_{b}-T_{b}T_{a}\in K(H^{2}(\mathbb{D})).$$ \label{thmDouglas}
\end{theorem}
Let $QC$ be the C*-algebra of functions in $H^{\infty}+C$ whose
complex conjugates also belong to $H^{\infty}+C$. Let us also
define the upper half-plane version of $QC$ as the following:
\begin{equation*}
    QC(\mathbb{R})=\{a\in L^{\infty}(\mathbb{R}):a\circ\mathfrak{C}^{-1}\in QC\}.
\end{equation*}
Going back and forth with Cayley transform one can deduce that
$QC(\mathbb{R})$ is a closed subalgebra of
$L^{\infty}(\mathbb{R})$.

 Let $scom(QC(\mathbb{T}))$ be the closed
ideal in $\mathcal{T}(QC(\mathbb{T}))$ generated by the
semi-commutators $\{T_{a}T_{b}-T_{ab}:a, b\in QC(\mathbb{T})\}$.
Then by Douglas' theorem, we have
\begin{equation*}
    com(\mathcal{T}(QC(\mathbb{T})))\subseteq
    scom(QC(\mathbb{T}))\subseteq K(H^{2}(\mathbb{D})) .
\end{equation*}
 By Proposition 7.12 of [Dou] we have
\begin{equation}
    com(\mathcal{T}(QC(\mathbb{T})))=scom(QC(\mathbb{T}))=K(H^{2}(\mathbb{D})) .
\end{equation}
 Now consider the symbol map
\begin{equation*}
 \Sigma:QC(\mathbb{T})\rightarrow\mathcal{T}(QC(\mathbb{T}))
\end{equation*}
  defined as
 $\Sigma(a)=T_{a}$. This map is linear but not necessarily multiplicative; however if we let $q$ be
 the quotient map
\begin{equation*}
    q:\mathcal{T}(QC(\mathbb{T})) \rightarrow \mathcal{T}(QC(\mathbb{T}))/scom(QC(\mathbb{T})) ,
\end{equation*}
then $q\circ\Sigma$ is multiplicative; moreover by equations (7)
and (8), we conclude that
 $q\circ\Sigma$ is an isometric *-isomorphism from $QC(\mathbb{T})$
 onto $\mathcal{T}(QC(\mathbb{T}))/K(H^{2}(\mathbb{D}))$.

\begin{definition}
Let $\varphi:\mathbb{D}\longrightarrow$ $\mathbb{D}$ or
$\varphi:\mathbb{H}\longrightarrow$ $\mathbb{H}$ be a holomorphic
self-map of the unit disc or the upper half-plane. The
\emph{composition
 operator} $C_{\varphi}$ on $H^{p}(\mathbb{D})$ or $H^{p}(\mathbb{H})$ with symbol $\varphi$ is defined by
\begin{equation*}
C_{\varphi}(g)(z)= g(\varphi(z)),\qquad
z\in\mathbb{D}\quad\textrm{or}\quad z\in\mathbb{H}.
\end{equation*}
\end{definition}

 Composition operators of the unit disc are always
 bounded [CM] whereas composition operators of the upper half-plane are not
 always bounded. For the boundedness problem of composition operators of the upper half-plane see
 [Mat].

The composition operator $C_{\varphi}$ on $H^{2}(\mathbb{D})$ is
carried over to
$(\frac{\tilde{\varphi}(z)+i}{z+i})C_{\tilde{\varphi}}$ on
$H^{2}(\mathbb{H})$ through $\Phi$, where $\tilde{\varphi}=$
$\mathfrak{C}\circ\varphi\circ\mathfrak{C}^{-1}$, i.e. we have
\begin{equation}
    \Phi C_{\varphi}\Phi^{-1} =
    T_{(\frac{\tilde{\varphi}(z)+i}{z+i})}C_{\tilde{\varphi}}.
\end{equation}

The Fourier transform $\mathcal{F}f$ of $f\in$
$\mathcal{S}(\mathbb{R})$ (the Schwartz space, for a definition
see [Rud, sec. 7.3, pp.~168]) is defined by
\begin{equation*}
(\mathcal{F}f)(t)=\hat{f}(t)=\frac{1}{\sqrt{2\pi}}\int_{-\infty}^{+\infty}e^{-itx}f(x)dx.
\end{equation*}
The Fourier transform extends to an invertible isometry from
$L^{2}(\mathbb{R})$ onto itself with inverse
\begin{equation*}
(\mathcal{F}^{-1}f)(t)=\frac{1}{\sqrt{2\pi}}\int_{-\infty}^{+\infty}
e^{itx}f(x)dx.
\end{equation*}
The following is a consequence of a theorem due to Paley and
Wiener [Koo, pp.~110--111]. Let $1 < p < \infty$. For $f\in$
$L^{p}(\mathbb{R})$, the following assertions are equivalent:
\begin{enumerate}
    \item[($i$)]  $f\in$ $H^{p}$,
    \item[($ii$)] $\textrm{supp}(\hat{f})\subseteq$ $[0,\infty)$
\end{enumerate}

A reformulation of the Paley-Wiener theorem says that the image of
$H^{2}(\mathbb{H})$ under the Fourier transform is
$L^{2}([0,\infty))$.

 By the Paley-Wiener theorem we observe that the operator
$$D_{\vartheta}=\Phi^{-1}\mathcal{F}^{-1}M_{\vartheta}\mathcal{F}\Phi$$
for $\vartheta\in C([0,\infty])$ maps $H^{2}(\mathbb{D})$ into
itself, where $C([0,\infty])$ denotes the set of continuous
functions on $[0,\infty)$ which have limits at infinity and $\Phi$
is the isometric isomorphism defined by equation (5). Since
$\mathcal{F}$ is unitary we also observe that
\begin{equation}
    \parallel D_{\vartheta}\parallel=\parallel M_{\vartheta}\parallel=\parallel\vartheta\parallel_{\infty}
\end{equation}
Let $F$ be defined as
\begin{equation}
F =\{D_{\vartheta}\in B(H^{2}(\mathbb{D})):\vartheta\in
C([0,\infty])\} .
\end{equation}
We observe that $F$ is a commutative C*-algebra with identity and
the map $D:C([0,\infty])\rightarrow F$ given by
\begin{equation*}
D(\vartheta)=D_{\vartheta}
\end{equation*}
is an isometric *-isomorphism by equation (10). Hence $F$ is
isometrically *-isomorphic to $C([0,\infty])$. The operator
$D_{\vartheta}$ is usually called a ``Fourier Multiplier.''

We observe that
\begin{equation}
T_{\frac{2i(1-z)}{2i+a(1-z)}}C_{\varphi_{a}}=D_{\vartheta_{a}}
\end{equation}
where
\begin{equation}
\varphi_{a}(z)=\frac{(2i-a)z+a}{-az+a+2i}
\end{equation}
 and
$$\vartheta_{a}(t)=e^{iat}$$.

\section{maximal ideal space of the Toeplitz-Composition
C*-algebra}

In this section we will analyze the structure of the C*-algebra
$C^{*}(\{C_{\varphi_{a}}\}\cup\mathcal{T}(PQC(\mathbb{T})))$
generated by the composition operator $C_{\varphi_{a}}$ and
Toeplitz operators with piece-wise quasi-continuous symbols where
$\varphi_{a}$ is as in equation (13) with $\Im(a)>0$. By equation
(12) and an easy application of Stone-Weierstrass theorem, this
C*-algebra is the same as $\Psi(PQC(\mathbb{T}),C([0,\infty]))$,
the C*-algebra generated by Toeplitz operators with piece-wise
quasi-continuous symbols and continuous Fourier multipliers.

   By the results of [Gul] and [Sch] the C*-algebra
   $\Psi(PQC(\mathbb{T}),C([0,\infty]))/K(H^{2})$ is commutative
   with identity. Hence it is of interest to characterize the
   maximal ideal space of this C*-algebra.

We will use the following theorem due to Power[Pow] in identifying
the maximal ideal space of
$\Psi(PQC(\mathbb{T}),C([0,\infty])/K(H^{2})$:
\begin{theorem} [\scshape Power's Theorem]
    Let $C_{1}$, $C_{2}$ and $C_{3}$ be three C*-subalgebras of $B(H)$ with identity,
    where $H$ is a separable Hilbert space, such that $M(C_{i})\neq$
    $\emptyset$, where $M(C_{i})$ is the space of multiplicative linear
    functionals of $C_{i}$, $i= 1,\,2,\,3$ and let $C$ be the C*-algebra that
    they generate. Then for the commutative C*-algebra $\tilde{C}=$
    $C/com(C)$ we have $M(\tilde{C})=$ $P(C_{1},C_{2},C_{3})\subset$
    $M(C_{1})\times M(C_{2})\times M(C_{3})$, where $P(C_{1},C_{2},C_{3})$ is defined to be
    the set of points $(x_{1},x_{2},x_{3})\in$ $M(C_{1})\times M(C_{2})\times M(C_{3})$
    satisfying the condition: \\
    \quad Given $0\leq a_{1} \leq 1$, $0 \leq a_{2} \leq 1$, $0 \leq a_{3} \leq 1$ $a_{1}\in C_{1}$, $a_{2}\in
    C_{2}$, $a_{3}\in C_{3}$
\begin{equation*}
    x_{i}(a_{i})=1\quad\textrm{with}\quad
        i=1,2,3\quad\Rightarrow\quad\| a_{1}a_{2}a_{3}\|=1.
\end{equation*}
    \label{thmpower}
\end{theorem}
The proof of this theorem can be found in [Pow]. We prove our main
theorem as follows:
\begin{thma}
Let
$$\Psi(PQC(\mathbb{T}),C([0,\infty])=C^{*}(\mathcal{T}(PQC(\mathbb{T}))\cup
F_{C([0,\infty])})=C^{*}(\mathcal{T}(PQC(\mathbb{T}))\cup\{C_{\varphi_{a}}\})$$
be the C*-algebra generated by Toeplitz operators with piece-wise
quasi-continuous symbols and Fourier multipliers on the Hardy
space $H^{2}(\mathbb{D})$. Then\\
$\Psi(PQC(\mathbb{T}),C([0,\infty])/K(H^{2})$ is a commutative
C*-algebra and its maximal ideal space is characterized as
\begin{equation*}
M(\Psi)\cong (M_{1}(\mathcal{T}(PC))\times M_{1}(QC)\times
[0,\infty])\cup
([\cup_{\lambda\in\mathbb{T}}(M_{\lambda}(\mathcal{T}(PC))\times
M_{\lambda}(QC))]\times\{\infty\})
\end{equation*}
where $M_{\lambda}(\mathcal{T}(PC))=\{x\in
M(\mathcal{T}(PC)):x\mid_{C(\mathbb{T})}=\delta_{\lambda},\delta_{\lambda}(T_{f})=f(\lambda)\}$
and $M_{\lambda}(QC)=\{x\in
M(QC):x\mid_{C(\mathbb{T})}=\delta_{\lambda},\delta_{\lambda}(T_{f})=f(\lambda)\}$
are the fibers of $M(\mathcal{T}(PC))$ and $M(QC)$ at $\lambda$
respectively.
\end{thma}
\begin{proof}
By the results of [Gul](Lemma 7) and [Sch](Lemma 2.0.15) we
already know that
$\Psi(PQC(\mathbb{T}),C([0,\infty])/K(H^{2}(\mathbb{D}))$ is a
commutative C*-algebra with identity. We will prove that the
characterization of its maximal ideal space is as above. We use
Power's theorem in our proof. Since for any ideal $I\subset A_{i}$
where $A_{i}$ for $i=1,2,3$ is a C*-algebra and $I$ is a closed
ideal we have $$C^{*}(A_{1}\cup A_{2}\cup A_{3})/I\cong
C^{*}((A_{1}/I)\cup (A_{2}/I)\cup (A_{3}/I)),$$ we take $H$ to be
such that $B(H^{2})/K(H^{2})\subset B(H)$,
$$ C_{1}=\mathcal{T}(PC(\mathbb{T}))/K(H^{2}),C_{2}=\mathcal{T}(QC(\mathbb{T}))/K(H^{2}),C_{3}=C^{*}(F_{C[0,\infty]}\cup K(H^{2}))/K(H^{2})$$
and
$$\quad\tilde{C}=\Psi(QC(\mathbb{T}),C([0,\infty]))/K(H^{2}(\mathbb{D})).$$
We have
\begin{equation*}
    M(C_{1})=M(\mathcal{T}(PC)),M(C_{2})=M(QC)\quad\textrm{and}\quad M(C_{3})=[0,\infty].
\end{equation*}
So we need to determine $(x,y,z)\in$ $M(\mathcal{T}(PC))\times
M(QC)\times [0,\infty]$ so
    that for all $a_{1}\in$ $PC(\mathbb{T})$,$a_{2}\in$ $QC(\mathbb{T})$ and $\vartheta\in$
    $C([0,\infty])$ with $0 < a,b, \vartheta\leq 1$, we have
\begin{equation*}
    \hat{T_{a_{1}}}(x)=\hat{a_{2}}(y)=\vartheta(z)=1\Rightarrow\parallel
    T_{a_{1}}T_{a_{2}}D_{\vartheta}\parallel_{e}=1\quad\textrm{or}\quad\parallel
    D_{\vartheta}T_{a_{1}}T_{a_{2}}\parallel_{e}=1.
\end{equation*}
For any $x\in$ $M(A)$ where
$A=\mathcal{T}(PC(\mathbb{T}))/K(H^{2})$ or
$A=QC(\mathbb{T})/K(H^{2})$ consider $\tilde{x}=$
    $x|_{C(\mathbb{T})}$ then $\tilde{x}\in$
    $M(C(\mathbb{T}))=$ $\mathbb{T}$. Hence
    $M(A)$ is fibered over $\mathbb{T}$, i.e.
\begin{equation*}
    M(A)=\bigcup_{\lambda\in\mathbb{T}}M_{\lambda},
\end{equation*}
    where
\begin{equation*}
    M_{\lambda}=\{x\in M(A):\tilde{x}=x|_{C(\mathbb{T})}=\delta_{\lambda}\}.
\end{equation*}
Let $x\in M(\mathcal{T}(PC))$ with $x\in M_{\lambda_{1}}$ and
$y\in M(QC)$ with $y\in M_{\lambda_{2}}$ such that
$\lambda_{1}\neq\lambda_{2}$. And let $z\in [0,\infty]$ be
arbitrary. Then there are functions $a_{1},a_{2}\in C(\mathbb{T})$
such that
$\hat{T_{a_{1}}}(x)=a_{1}(\lambda_{1})=\hat{T_{a_{2}}}(x)=a_{2}(\lambda_{2})=1$
and $\parallel a_{1}a_{2}\parallel_{\infty}<1$. Since
$$\parallel T_{a_{1}}T_{a_{2}}\parallel_{e}=\parallel T_{a_{1}a_{2}}\parallel_{e}$$ and $$\parallel
T_{a_{1}a_{2}}\parallel_{e}=\parallel
a_{1}a_{2}\parallel_{\infty}$$ we have $$\parallel
T_{a_{1}}T_{a_{2}}\parallel_{e}=\parallel
a_{1}a_{2}\parallel_{\infty}<1$$ So for $\vartheta(t)\equiv 1$ we
have $\vartheta(z)=1$ and $$\parallel
T_{a_{1}}T_{a_{2}}D_{\vartheta}\parallel_{e}=\parallel
T_{a_{1}}T_{a_{2}}\parallel_{e}=\parallel
a_{1}a_{2}\parallel_{\infty}<1.$$ Hence $(x,y,z)\not\in M(\Psi)$.
So if $(x,y,z)\in M(\Psi)$ and $x\in M_{\lambda_{1}}$, $y\in
M_{\lambda_{2}}$ then $\lambda_{1}=\lambda_{2}$.

Now let $x\in M(\mathcal{T}(PC))$ and $y\in M(QC)$ with $x,y\in
M_{\lambda}$ such that $\lambda\neq 1$. And let $z\in [0,\infty]$
with $z\neq\infty$. Then there is $a\in C(\mathbb{T})$ and
$\vartheta\in C([0,\infty])$ such that
$a(\lambda)=\vartheta(z)=1$. Using the isometric isomorphism
$\Phi$ introduced by equation (5) we have $$\parallel
T_{a}D_{\vartheta}\parallel=\parallel\Phi^{-1}\circ
(T_{a_{2}}\tilde{D}_{\vartheta})\circ\Phi\parallel=\parallel
T_{a_{2}}\tilde{D}_{\vartheta}\parallel_{H^{2}(\mathbb{H})}$$
where $a_{2}=a\circ\mathfrak{C}$ and
$\tilde{D}_{\vartheta}=\Phi\circ D_{\vartheta}\circ\Phi^{-1}$. Let
$a$ and $\vartheta$ have compact supports, let $1$ be out of the
support of $a$ and let $\tilde{\vartheta}$ be
\begin{equation*}
    \tilde{\vartheta}(w)=
        \begin{cases}
            \vartheta(w)\qquad\textrm{if}\qquad w\geq 0\\
            \vartheta(-w)\qquad\textrm{if}\qquad w < 0
        \end{cases}
\end{equation*}
Since $\lambda\neq 1$ and $1$ is out of the support of $a$,
$a_{2}$ has also compact support in $\mathbb{R}$. We have
\begin{equation*}
    PM_{a_{2}}\tilde{D}_{\tilde{\vartheta}}|_{H^{2}}=T_{a_{2}}\tilde{D}_{\vartheta},
\end{equation*}
    where $P:L^{2}\rightarrow$ $H^{2}$ is the orthogonal projection of
    $L^{2}(\mathbb{R})$ onto $H^{2}(\mathbb{H})$. So we have
\begin{equation*}
    \parallel T_{a_{2}}\tilde{D}_{\vartheta}\parallel_{H^{2}(\mathbb{H})}\leq\parallel M_{a_{2}}\tilde{D}_{\tilde{\vartheta}}\parallel_{L^{2}(\mathbb{R})}.
\end{equation*}
     By a result of
    Power (see [Pow2] and [Sch]) under these conditions we have
    \begin{equation}
        \parallel M_{a_{2}}\tilde{D}_{\tilde{\vartheta}}\parallel_{L^{2}}< 1\Rightarrow\parallel
T_{a_{1}}T_{a_{2}}\tilde{D}_{\vartheta}\parallel_{H^{2}}< 1
    \end{equation}
where $a_{1}\equiv 1$. Hence we have $(x,y,z)\not\in M(\Psi)$.

So if $(x,y,z)\in M(\tilde{C})$, then either $z= \infty$ or
$x,y\in M_{1}$.

Let $z= \infty$ and $x\in M(PC)$, $y\in M(QC)$. Let $a_{1}\in PC$,
$a_{2}\in QC$ and $\vartheta\in C([0,\infty])$ such that
\begin{equation*}
    0\leq a_{1},a_{2},\vartheta\leq
    1\quad\textrm{and}\quad\hat{T_{a_{1}}}(x)=\hat{a_{2}}(y)=\vartheta(z)= 1.
\end{equation*}
    Consider
    \begin{eqnarray}
& &\parallel\tilde{D}_{\vartheta}T_{a_{1}\circ\mathfrak{C}}T_{a_{2}\circ\mathfrak{C}}\parallel_{H^{2}(\mathbb{H})e}=\parallel\mathcal{F}\tilde{D}_{\vartheta}T_{(a_{1}\circ\mathfrak{C})(a_{2}\circ\mathfrak{C})}\mathcal{F}^{-1}\parallel_{L^{2}([0,\infty))e}\nonumber\\
    & &   =\parallel M_{\vartheta}\mathcal{F}T_{(a_{1}\circ\mathfrak{C})(a_{2}\circ\mathfrak{C})}\mathcal{F}^{-1}\parallel_{L^{2}([0,\infty))e}\nonumber\\
    & & =\parallel M_{\vartheta}\mathcal{F}(\mathcal{F}^{-1}M_{\chi_{[0,\infty)}}\mathcal{F})M_{(a_{1}\circ\mathfrak{C})(a_{2}\circ\mathfrak{C})}\mathcal{F}^{-1}\parallel_{L^{2}([0,\infty))e}\nonumber\\
    & & =\parallel M_{\vartheta}\mathcal{F}M_{(a_{1}\circ\mathfrak{C})(a_{2}\circ\mathfrak{C})}\mathcal{F}^{-1}\parallel_{L^{2}([0,\infty))e}.
    \end{eqnarray}
    Fix a compact operator $K\in K(H^{2})$. Since $x,y\in M_{\lambda}$ are on the same fiber we have $\parallel a_{1}a_{2}\parallel_{\infty}=1$. Hence it
    is possible to choose $g\in$ $L^{2}([0,\infty))$ with $\parallel g\parallel_{L^{2}([0,\infty))}=1$ such that
\begin{equation*}
    \parallel(\mathcal{F}M_{(a_{1}\circ\mathfrak{C})(a_{2}\circ\mathfrak{C})}\mathcal{F}^{-1})g\parallel\geq
    1-\varepsilon
\end{equation*}
    for given $\varepsilon > 0$. Since
    $\vartheta(\infty)= 1$ there exists $w_{0} > 0$ so that
\begin{equation*}
    1-\varepsilon\leq\vartheta(w)\leq 1\quad\forall w\geq w_{0}.
\end{equation*}
    Let $t_{0}\geq w_{0}$ such that $\parallel K S_{-t_{0}}g\parallel\leq\varepsilon$. Since the support of
    $(S_{-t_{0}}\mathcal{F}M_{(a_{1}\circ\mathfrak{C})(a_{2}\circ\mathfrak{C})}\mathcal{F}^{-1})g$ lies in
    $[t_{0},\infty)$ where $S_{t}$ is the translation by $t$, we have
    \begin{eqnarray}
    & &\parallel M_{\vartheta}(S_{-t_{0}}\mathcal{F}M_{(a_{1}\circ\mathfrak{C})(a_{2}\circ\mathfrak{C})}\mathcal{F}^{-1})g+K S_{-t_{0}}g\parallel_{2}\\
    \geq&\nonumber &\inf\{\vartheta(w): w\in (w_{0},\infty)\}\parallel(\mathcal{F}M_{(a_{1}\circ\mathfrak{C})(a_{2}\circ\mathfrak{C})}\mathcal{F}^{-1})g\parallel_{2}-\varepsilon\geq(1-\varepsilon)^{2}-\varepsilon
    \end{eqnarray}
    Since
\begin{equation*}
    S_{-t_{0}}\mathcal{F}M_{(a_{1}\circ\mathfrak{C})(a_{2}\circ\mathfrak{C})}\mathcal{F}^{-1}=\mathcal{F}M_{(a_{1}\circ\mathfrak{C})(a_{2}\circ\mathfrak{C})}\mathcal{F}^{-1}S_{-t_{0}}
\end{equation*}
and $S_{-t_{0}}$ is an isometry on $L^{2}([0,\infty))$ by
equations
    (15) and (16), we conclude that
\begin{equation*}
    \parallel M_{\vartheta}\mathcal{F}M_{(a_{1}\circ\mathfrak{C})(a_{2}\circ\mathfrak{C})}\mathcal{F}^{-1}\parallel_{L^{2}([0,\infty))e}=\parallel
    \tilde{D}_{\vartheta}T_{a_{1}\circ\mathfrak{C}}T_{a_{2}\circ\mathfrak{C}}\parallel_{H^{2}e}= 1
\end{equation*}
which implies that $(x,y,\infty)\in M(\Psi)$ $\forall x\in
M(\mathcal{T}(PC))$ and $y\in M(QC)$ with $x,y\in M_{\lambda}$.

Now let $x\in M_{1}(\mathcal{T}(PC))$, $y\in M_{1}(QC)$ and $z\in
[0,\infty]$. Let
    $a_{1}\in PC$, $a_{2}\in QC$ and $\vartheta\in C([0,\infty])$ such that
\begin{equation*}
    \hat{T_{a_{1}}}(x)=\hat{a_{2}}(y)=\vartheta(z)=1\quad\textrm{and}\quad
    0\leq a_{1},a_{2},\vartheta\leq 1.
\end{equation*}
We have two cases: Either $a_{1}$ is continuous at $1$ or $a_{1}$
is not continuous at $1$. Suppose $a_{1}$ is continuous at $1$.

    By a result of Sarason (see [Sar1]
    lemmas 5 and 7) for a given $\varepsilon > 0$ there is a $\delta> 0$ so that
    \begin{equation}
    \mid\hat{a_{2}}(y)-\frac{1}{2\delta}\int_{-\delta}^{\delta}a_{2}(e^{i\theta})d\theta\mid\leq\varepsilon.
    \end{equation}
    Since $\hat{a_{2}}(y)= 1$ and $0\leq a_{2}\leq 1$, this implies
    that for all $\varepsilon > 0$ there exists $w_{0} > 0$ such that
    $\sqrt{1-\varepsilon}\leq$ $a_{2}\circ\mathfrak{C}(w)\leq 1$ for a.e.\quad$w$ with $\mid
    w\mid > w_{0}$. Since $a_{1}$ is continuous at $1$ we also have $\sqrt{1-\varepsilon}\leq$ $a_{1}\circ\mathfrak{C}(w)\leq 1$ for a.e.\quad$w$ with $\mid
    w\mid > w_{0}$. Hence we have $1-\varepsilon\leq$ $(a_{1}\circ\mathfrak{C}a_{2}\circ\mathfrak{C})(w)\leq 1$ for a.e.\quad$w$ with $\mid
    w\mid > w_{0}$. Let $\tilde{\vartheta}$ be
\begin{equation*}
\tilde{\vartheta}(w)=
    \begin{cases}
    \vartheta(w)\qquad\textrm{if}\qquad w\geq 0\\
    0\qquad\textrm{if}\qquad w < 0
    \end{cases} .
\end{equation*}
Then we have
\begin{equation}
\tilde{D}_{\vartheta}T_{(a_{1}\circ\mathfrak{C})(a_{2}\circ\mathfrak{C})}=\tilde{D}_{\tilde{\vartheta}}M_{(a_{1}\circ\mathfrak{C})(a_{2}\circ\mathfrak{C})}.
\end{equation}
    Let $\varepsilon > 0$ be given. Let
    $g\in$ $H^{2}$ so that $\parallel g\parallel_{2}= 1$ and
    $\parallel D_{\tilde{\vartheta}}g\parallel_{2}\geq 1-\varepsilon$.
    Since $1-\varepsilon\leq$ $(a_{1}\circ\mathfrak{C}a_{2}\circ\mathfrak{C})(w)\leq 1$ for a.e.\quad$w$ with $\mid
    w\mid > w_{0}$, there is a $w_{1} > 2w_{0}$ so that
\begin{equation*}
\parallel S_{w_{1}}g-M_{(a_{1}\circ\mathfrak{C})(a_{2}\circ\mathfrak{C})}S_{w_{1}}g\parallel_{2}\leq 2\varepsilon.
\end{equation*}
Let $K\in K(H^{2})$ and let $w_{1} > 2w_{0}$ so that $\parallel K
S_{w_{1}}g\parallel\leq\varepsilon$.
    We have $\parallel\tilde{D}_{\tilde{\vartheta}}\parallel = 1$ and this implies that
    \begin{equation}
        \parallel\tilde{D}_{\tilde{\vartheta}}S_{w_{1}}g-\tilde{D}_{\tilde{\vartheta}}M_{(a_{1}\circ\mathfrak{C})(a_{2}\circ\mathfrak{C})}S_{w_{1}}g-K S_{w_{1}}g\parallel_{2}\leq 3\varepsilon .
    \end{equation}
    Since $S_{w}\tilde{D}_{\tilde{\vartheta}}= \tilde{D}_{\tilde{\vartheta}}S_{w}$ and
    $S_{w}$ is unitary for all $w\in \mathbb{R}$, we have
    $$\parallel\tilde{D}_{\tilde{\vartheta}}M_{(a_{1}\circ\mathfrak{C})(a_{2}\circ\mathfrak{C})}S_{w_{1}}g+K S_{w_{1}}g\parallel_{2}\geq 1-4\varepsilon$$
    and this implies that $$\parallel
    D_{\vartheta}T_{a_{1}}T_{a_{2}}\parallel_{e}=1$$
    Now let $a_{1}$ be discontinuous at $1$. Then since $x=(1,t)$
    where
    $\hat{T_{a_{1}}}(x)=\hat{T_{a_{1}}}((1,t))=ta_{1}(1^{-})+(1-t)a_{1}(1^{+})=1$
    and $\hat{T_{a_{1}}}(x)=1$, if $0<t<1$ then $a_{1}(1^{-})=a_{1}(1^{+})$ (since $0\leq a_{1}\leq 1$), this implies that $a_{1}$ is
    continuous at $1$ and this contradicts our assumption. So
    $t=1$ or $t=0$. This implies that $a_{1}(1^{-})=1$ or
    $a_{1}(1^{+})=1$. If $a_{1}(1^{+})=1$ then the same arguments
    leading to equations (18) and (19) apply and we obtain
    $$\parallel D_{\vartheta}T_{a_{1}}T_{a_{2}}\parallel_{e}=1.$$
    If $a_{1}(1^{-})=1$ then we proceed by the same arguments by
    replacing $S_{w_{1}}$ with $S_{-w_{1}}$ and similarly we obtain
    $$\parallel D_{\vartheta}T_{a_{1}}T_{a_{2}}\parallel_{e}=1.$$

\end{proof}

\section{essential spectra of linear combinations of Toeplitz and
composition operators}

In this last section we give an application of the main theorem
above to the problem of determining the essential spectra of
linear combinations of certain class of Toeplitz and composition
operators.

In [Gul] we showed that if
$\varphi:\mathbb{D}\rightarrow\mathbb{D}$ is of the following form
\begin{equation}
\varphi(z)=\frac{2iz+\eta(z)(1-z)}{2i+\eta(z)(1-z)}
\end{equation}
where $\eta\in$ $H^{\infty}(\mathbb{D})$ with $\Im(\eta(z)) >
\epsilon> 0$ for all $z\in$ $\mathbb{D}$, then
$\exists\alpha\in\mathbb{R}^{+}$ such that
\begin{equation}
    C_{\varphi} = T_{\frac{2i+\eta(z)(1-z)}{2i}}\sum_{n=0}^{\infty}T_{(i\alpha-\eta(z))^{n}}D_{\vartheta_{n}},
\end{equation}
where $\vartheta_{n}(t)=$ $\frac{(-it)^{n}e^{-\alpha t}}{n!}$. By
equation (21) we deduce that if $\eta\in QC$ then
$C_{\varphi}\in\Psi(PQC(\mathbb{T}),C[0,\infty])$. Using this fact
we can extract a full characterization of essential spectra of
operators $T_{a}C_{\varphi}$ and $C_{\varphi}+T_{a}$ where
$\varphi$ is as in equation (20) with $\eta\in QC$ and $a\in PC$.
We formulate our result as the following theorem:
\begin{theorem}
Let $a\in PC(\mathbb{T})$ and $\eta\in QC(\mathbb{T})$ with
$\Im(\eta(z)) > \epsilon> 0$ for all $z\in$ $\mathbb{D}$. Let
$\varphi:\mathbb{D}\rightarrow\mathbb{D}$ be an analytic self-map
of $\mathbb{D}$ of the following form
$$\varphi(z)=\frac{2iz+\eta(z)(1-z)}{2i+\eta(z)(1-z)}$$
then for $C_{\varphi}:H^{2}(\mathbb{D})\rightarrow
H^{2}(\mathbb{D})$ and $T_{a}:H^{2}(\mathbb{D})\rightarrow
H^{2}(\mathbb{D})$ we have
\begin{itemize}
\item (i)\quad $\sigma_{e}(T_{a}C_{\varphi})=\\
\overline{\{(sa(1^{-})+(1-s)a(1^{+}))e^{izt}:z\in\mathcal{C}_{1}(\eta),t\in
[0,\infty]\quad\textrm{and}\quad s\in [0,1]\}}$
\item (ii)\quad$\sigma_{e}(T_{a}+C_{\varphi})=\\
\overline{\{(sa(\lambda^{-})+(1-s)a(\lambda^{+}))+e^{izt}:z\in\mathcal{C}_{1}(\eta),t\in
[0,\infty],\lambda\in\mathbb{T}\quad\textrm{and}\quad s\in
[0,1]\}}$
\end{itemize}
where $\mathcal{C}_{1}(\eta)$ of $\eta\in
  H^{\infty}$ is the set of cluster points that is defined as the set of points
  $z\in$ $\mathbb{C}$ for which there is a sequence $\{z_{n}\}\subset$
  $\mathbb{D}$ so that $z_{n}\rightarrow$ $1$ and
  $\eta(z_{n})\rightarrow$ $z$, $a(\lambda^{+})=\lim_{\theta\rightarrow 0^{+}}a(\lambda
e^{i\theta})$ and $a(\lambda^{-})=\lim_{\theta\rightarrow
0^{-}}a(\lambda e^{i\theta})$.
\end{theorem}

\begin{proof}
Let
$\Psi=\Psi(PQC(\mathbb{T}),C([0,\infty]))/K(H^{2}(\mathbb{D}))$.
Since $\Psi$ is C*-subalgebra of the Calkin algebra
$B(H^{2})/K(H^{2})$ and $[T_{a}C_{\varphi}]\in\Psi$, by equation
(2) we have
\begin{equation}
\sigma_{e}(T_{a}C_{\varphi})=\sigma_{B(H^{2})/K(H^{2})}([T_{a}C_{\varphi}])=\sigma_{\Psi}([T_{a}C_{\varphi}]).
\end{equation}
Since $\Psi$ is a commutative C*-algebra, by equation (3) we have
\begin{equation}
\sigma_{\Psi}([T_{a}C_{\varphi}])=\{\hat{[T_{a}C_{\varphi}]}(\phi):\phi\in
M(\Psi)\}
\end{equation}
 By the Main Theorem above, for any $\phi\in M(\Psi)$ we
have $\phi=(x,y,\infty)\in M(\mathcal{T}(PC))\times M(QC)\times
[0,\infty]$ where $x,y\in M_{\lambda}$ or $\phi=(x,y,z)\in
M(\mathcal{T}(PC))\times M(QC)\times [0,\infty]$ where $x,y\in
M_{1}$ and $z\in [0,\infty]$.

Let $\phi=(x,y,z)\in M(\mathcal{T}(PC))\times M(QC)\times
[0,\infty]$ with $x,y\in M_{\lambda}$, $\lambda\neq 1$ and
$z=\infty$. Then by equation (21) we have
\begin{eqnarray}
 &\nonumber&
 \hat{[T_{a}C_{\varphi}]}(\phi)=\phi([T_{a}T_{\frac{2i+\eta(z)(1-z)}{2i}}\sum_{n=0}^{\infty}T_{(i\alpha-\eta(z))^{n}}D_{\vartheta_{n}}])=\\
 & &\sum_{n=0}^{\infty}x(T_{a})y(T_{\frac{2i+\eta(z)(1-z)}{2i}})y(T_{(i\alpha-\eta(z))^{n}})\vartheta_{n}(\infty).
\end{eqnarray}
Since $\vartheta_{n}(\infty)=0$ $\forall n\in\mathbb{N}$ we have
$$\hat{[T_{a}C_{\varphi}]}(\phi)=0$$ in this case.

Now let $\phi=(x,y,z)\in M(\mathcal{T}(PC))\times M(QC)\times
[0,\infty]$ with $x,y\in M_{1}$ and $z\in [0,\infty]$. We observe
that since $y\in M_{1}$,
$\hat{[T_{\frac{2i+\eta(z)(1-z)}{2i}}]}(y)=1$ and hence by
equation (24) we have
\begin{eqnarray}
 &\nonumber &\hat{[T_{a}C_{\varphi}]}(\phi)=\sum_{n=0}^{\infty}x([T_{a}])(i\alpha-\hat{\eta}(y))^{n}\frac{(-iz)^{n}e^{-\alpha z}}{n!}=\\
 & &x(T_{a})e^{-\alpha z}\sum_{n=0}^{\infty}\frac{((-iz)(i\alpha-\hat{\eta}(y)))^{n}}{n!}=x(T_{a})e^{iz\hat{\eta}(y)}
\end{eqnarray}
By a result of Shapiro [Sha] we have for any
$\lambda\in\mathbb{T}$, $$\{\hat{\eta}(y):y\in
M_{\lambda}(QC)\}=\mathcal{C}_{\lambda}(\eta)$$ since $\eta\in
QC$. Since any $x\in M_{1}(\mathcal{T}(PC))$ is of the form
$x=(1,s)$ with $(1,s)([T_{a}])=sa(1^{-})+(1-s)a(1^{+})$ where
$s\in [0,1]$ we have $x([T_{a}])=sa(1^{-})+(1-s)a(1^{+})$ where
$a(1^{+})=\lim_{\theta\rightarrow 0^{+}}a(e^{i\theta})$ and
$a(1^{-})=\lim_{\theta\rightarrow 0^{-}}a(e^{i\theta})$. Combining
these with equation (25) the proof of (i) follows.

The proof of (ii) follows by a very similar manner: Like in
equation (24) we have
$$\hat{[T_{a}+C_{\varphi}]}(\phi)=\phi([T_{a}])+\phi([C_{\varphi}])=x([T_{a}])+\sum_{n=0}^{\infty}y([T_{\frac{2i+\eta(z)(1-z)}{2i}}])y([T_{(i\alpha-\eta(z))^{n}}])\vartheta_{n}(z)$$
where $\phi=(x,y,z)\in M(\mathcal{T}(PC))\times M(QC)\times
[0,\infty]$. For $z=\infty$ we have $\phi([C_{\varphi}])=0$. If
$x,y\in M_{1}$ then as in equation (25) together with the result
of Shapiro [Sha], $\phi([C_{\varphi}])=e^{i\mu z}$ for some
$\mu\in\mathcal{C}_{1}(\eta)$ and $z\in [0,\infty]$. As in
equations (22) and (23) we have
$$\sigma_{e}(T_{a}+C_{\varphi})=\{\hat{[T_{a}+C_{\varphi}]}(\phi):\phi\in
M(\Psi)\}$$ and the proof of (ii) follows.
\end{proof}
\bibliographystyle{amsplain}

\end{document}